\documentclass{aims}

\usepackage{amssymb}
\usepackage{amsmath}
\usepackage{paralist}
\usepackage[misc]{ifsym}
\usepackage{epsfig}
\usepackage{epstopdf}
\usepackage{caption}
\usepackage{subcaption}
\usepackage[colorlinks=true]{hyperref}
\hypersetup{urlcolor=blue, citecolor=red}
\usepackage{algorithm}
\usepackage{algorithmic}
\usepackage{tikz}
\usepackage{CJK}
\usepackage{mathtools}
\usepackage{amsthm,lipsum}
\usepackage{mathrsfs}
\usepackage{cases}
\usepackage{cite}

\hypersetup{urlcolor=blue, citecolor=red}
\allowdisplaybreaks

\def\currentvolume{X}
 \def\currentissue{X}
  \def\currentyear{200X}
   \def\currentmonth{XX}
    \def\ppages{X-XX}
     \def\DOI{10.3934/xx.xxxxxxx}

\newtheorem{theorem}{Theorem}[section]

\newtheorem{lemma}[theorem]{Lemma}

\theoremstyle{definition}
\newtheorem{definition}[theorem]{Definition}

\title[Revisited for existence proof of ...]
{Revisited for existence proof of optimal solution in Bernoulli free boundary problem using an energy-gap cost functional} 

\author[Shiouhe Wang, Fang Shen, Yi Yang, Xueshang Feng and Jiansen He]{}


\thanks{$^*$Corresponding author: Fang Shen}

\begin{document}
\maketitle

\centerline{\parbox{\textwidth}{\centering\scshape
    Shiouhe Wang$^{{\href{mailto:xhwang@spaceweather.ac.cn}{\textrm{\Letter}}}1,2}$,\ 
    Fang Shen$^{{\href{mailto:fshen@spaceweather.ac.cn}{\textrm{\Letter}}}*1,2}$,\ 
    Yi Yang$^{{\href{mailto:yyang@spaceweather.ac.cn}{\textrm{\Letter}}}1,2}$ \\
    Xueshang Feng$^{{\href{mailto:fengx@spaceweather.ac.cn}{\textrm{\Letter}}}1}$\ and\ 
    Jiansen He$^{{\href{mailto:jshept@qq.com}{\textrm{\Letter}}}3}$
  }}

\medskip

{\footnotesize
 \begin{center}

$^1$SIGMA Weather Group, State Key Laboratory of Space Weather,\\

National Space Science Center, Chinese Academy of Sciences,\\

Beijing 100190, People's Republic of China

\end{center}
} 

\medskip

{\footnotesize
 \begin{center}

$^2$College of Earth and Planetary Sciences,University of Chinese Academy of Sciences,\\

Beijing 100049, People's Republic of China

\end{center}
}

\medskip

{\footnotesize

\begin{center}

$^3$School of Earth and Space Sciences, Peking University, \\

Beijing 100871, People’s Republic of China

\end{center}

}

\bigskip

 \centerline{(Communicated by Handling Editor)}


\begin{abstract}
Bernoulli free boundary problem is numerically solved via shape optimization that minimizes a cost functional subject to state problems constraints.
In \cite{1}, an energy-gap cost functional was formulated based on two auxiliary state problems,
with existence of optimal solution attempted through continuity of state problems with respect to the domain.
Nevertheless, there exists a corrigendum in Eq.(48) in \cite{1},
where the boundedness of solution sequences for state problems with respect to the domain cannot be directly estimated via the Cauchy-Schwarz inequality as \textbf{Claimed}.
In this comment, we rectify this proof by Poincaré-Friedrichs inequality.
\end{abstract}


\section{Corrigendum}
Firstly, we show the Bernoulli free boundary problem and its corresponding shape optimization problem.
Bernoulli free boundary problem is formulated as follows:
\begin{definition} (Bernoulli free boundary problem): 
Let $\Omega_{0}$ and $\Omega_{1}$ be bounded and connected domains in $\mathbb{R}^{2}$ such that $\overline{\Omega_{1}}\subset\Omega_{0}$, 
and define $\Omega\coloneqq\Omega_{0}\setminus\overline{\Omega_{1}}$ possessing a moving boundary $\partial\Omega_{0}\coloneqq\Sigma$ and a fixed boundary $\partial\Omega_{1}\coloneqq\Gamma$. 
The exterior Bernoulli free boundary problem is considered as an overdetermined boundary value problem for a given positive constant $\lambda$, 
which implies that finding $u\in H^1{\left(\Omega\right)}$ and unknown external boundary $\Sigma$ satisfy:
  \begin{equation} 
  \begin{cases}
  -\Delta u=0 & \text{in $\Omega$,} \\
  u=1 & \text{on $\Gamma$,} \\
  u=0 & \text{on $\Sigma$,} \\
  -\partial_{\mathbf{n}} u=\lambda & \text{on $\Sigma$,} \\
  \end{cases} 
  \label{con:Eq.1}
  \end{equation}
where $\partial_{\mathbf{n}} u\coloneqq \nabla u \cdot \mathbf{n}$ and $\mathbf{n}$ is unit normal vector directed into the complement of $\Omega$.
\label{Bernoulli free boundary problem}
\end{definition}
The existence of solution for Eq.(\ref{con:Eq.1}) has been established by variational method in \cite{2}. 
Bernoulli free boundary problem is numerically solved via shape optimization that minimizes a cost functional subject to state problems constraints.
In \cite{1}, an energy-gap cost functional was formulated based on two auxiliary state problems as follows:
\begin{definition} (Shape Optimization):
The shape optimization requires to minimize the cost functional $\mathcal{J}\left(\Omega\right)$ over admissible domains $\mathcal{O}_{ad}$. 
The cost functional $\mathcal{J}$ is formulated as:
$\mathcal{J}\left(\Omega\right)=\int_{\Omega}\left|\nabla\left(u_{N}-u_{R}\right)\right|^{2}dx$,
where
\begin{subequations}\label{con:Eq.2}
     \begin{numcases}{\mbox{$u_{N},u_{R}$ is the solution of }~:}
         \text{$u_{N}\in H^1\left(\Omega\right)$}  
         \begin{cases}
          -\Delta u_{N}=0 & \text{in $\Omega$,} \\
          u_{N}=1 & \text{on $\Gamma$,} \\
          \partial_{\mathbf{n}} u_{N}=\lambda & \text{on $\Sigma$,} \\
          \end{cases}
          \label{con:Eq.2.1}\\
          \text{$u_{R}\in H^1\left(\Omega\right)$}   \label{con:Eq.2.2}
         \begin{cases}
          -\Delta u_{R}=0 & \text{in $\Omega$,} \\
          u_{R}=1 & \text{on $\Gamma$,} \\
          \partial_{\mathbf{n}} u_{R}+\beta u_{R} =\lambda & \text{on $\Sigma$,} \\
         \end{cases} 
     \end{numcases}
\end{subequations}
where $\beta$ is a positive constant. 

\label{Shape optimization problem}
\end{definition}
The definition of admissible domains $\mathcal{O}_{ad}$ can be found in references \cite{1,3}.
The variational formulations of two auxiliary state problems are formulated the following lemma:
\begin{lemma} \rm (Variational Formulations):
The variational formulations of Eq.\ref{con:Eq.2.1} and Eq.\ref{con:Eq.2.2} can be expressed as follows:
  \begin{equation} 
  \begin{cases}
  \text{find $w_{N}=u_{N}-u_{N_{0}}\in H_{\Gamma}^{1}\left(\Omega\right)$, such that}\\
  a\left(w_{N},\phi\right)=\int_{\Sigma}\lambda\phi d\sigma-\int_{\Omega}\nabla u_{N_{0}}\nabla \phi dx, \forall \phi\in H_{\Gamma}^{1}\left(\Omega\right),\\
  \end{cases} 
  \label{con:Eq.3.1}
  \end{equation}
  \begin{equation} 
  \begin{cases}
  \text{find $w_{R}=u_{R}-u_{R_{0}}\in H_{\Gamma}^{1}\left(\Omega\right)$, such that}\\
  a\left(w_{R},\phi\right)+\beta a_{\Sigma}\left(w_{R},\phi\right)=\int_{\Sigma}\lambda\phi d\sigma-\int_{\Omega}\nabla u_{R_{0}}\nabla \phi dx, \forall \phi\in H_{\Gamma}^{1}\left(\Omega\right), \\
  \end{cases} 
  \label{con:Eq.3.2}
  \end{equation}
where $H_{\Gamma}^{1}\left(\Omega\right)=\left\{v\in H^{1}\left(\Omega\right):v|_{\Gamma}=0\right\}$ equipped with a norm
\begin{equation} 
  \Vert v\Vert_{H_{\Gamma}^{1}\left(\Omega\right)}=\left(\int_{\Omega}\left|\nabla v\right|^2dx\right)^{\frac{1}{2}},
\label{con:Eq.5}
\end{equation}
and $u_{N_{0}},u_{R_{0}}$ are two fixed functions in $H^{1}\left(U\right)$ such that $u_{N_{0}}=u_{R_{0}}=1$.
$U$ is a fixed, connected and bounded open subset such that $\forall \Omega \in \mathcal{O}_{ad}, \Omega\subset U$.
$a\left(\cdot,\cdot\right),a_{\Sigma}\left(\cdot,\cdot\right)$ are bilinear functionals on $H_{\Gamma}^{1}\left(\Omega\right)\times H_{\Gamma}^{1}\left(\Omega\right)$ by these definitions:
\begin{equation} 
  a\left(w,v\right)=\int_{\Omega}\nabla w\nabla v dx,
\label{con:Eq.6}
\end{equation}
\begin{equation} 
  a_{\Sigma}\left(w,v\right)=\int_{\Sigma}\gamma_{0}\left(w\right)\cdot \gamma_{0}\left(v\right) dx,
\label{con:Eq.7}
\end{equation}
where $w,v\in H_{\Gamma}^{1}\left(\Omega\right)$, $\gamma_{0}:H^{1}\left(\Omega\right)\rightarrow H^{\frac{1}{2}}\left(\Omega\right)$.
\label{Variational Formulations}
\end{lemma}
The existence and uniqueness of solutions for two state problems(Eq.\ref{con:Eq.3.1},Eq.\ref{con:Eq.3.2}) are verified by Lax-Milgram theorem.
To prove continuity of the state solutions $u_{N}$, $u_{R}$ with respect to the domain $\Omega$, \cite{1} defines the graph $\mathcal{F}$ as:
\begin{equation} 
  \mathcal{F}=\left\{\left(\Omega,u_{N}\left(\Omega\right),u_{R}\left(\Omega\right)\right):\Omega\in\mathcal{O}_{ad},\text{and $u_{N}, u_{R}$ satisfies Eq.\ref{con:Eq.3.1}, Eq.\ref{con:Eq.3.2} on $\Omega$.}\right\},
\label{con:Eq.8}
\end{equation}
Subsequently, the shape optimization \ref{Shape optimization problem} is represented as:
\begin{equation} 
    \underset{\mathcal{F}}{\min}J\left(\Omega,u_{N}\left(\Omega\right),u_{R}\left(\Omega\right)\right).
    \label{con:Eq.9}
\end{equation}
The existence of optimal solution for the shape optimization problem Eq.\ref{con:Eq.9} reduces to proving compactness of $\mathcal{F}$ and lower semi-continuity of $\mathcal{J}$.
However, in \cite{1}'s proof for the compactness of $\mathcal{F}$, 
there exists a corrigendum in the boundedness for the sequence ${u_{R_k}}$ (Eq.48 in \cite{1}).
Specifically, erroneous step of Eq.(48) in \cite{1}:\\
\quad Let $\phi=u_{R_{k}}\in H^{1}\left(\Omega_{k}\right)$, then variational formulation Eq.\ref{con:Eq.3.1} is:
\begin{equation}
  \int_{\Omega_{k}}\nabla u_{R_{k}}\cdot \nabla u_{R_{k}}dx=-\int_{\Omega_{k}}\nabla u_{R_{0}}\cdot \nabla u_{R_{k}}dx-\int_{\Sigma_{k}}\beta \left|u_{R_{k}}\right|^{2}d\sigma+\int_{\Sigma_{k}}\lambda u_{R_{k}}d\sigma,
  \label{con:Eq.10}
\end{equation}
where $\beta,\lambda>0$, and solution sequence $u_{R_{k}}$ was constructed in \cite{1} for verifying the compactness of $\mathcal{J}$ based on weak convergence of solution sequences in $H^{1}\left(U\right)$.
Then, \cite{1} uses the Cauchy-Buniakowsky-Schwarz inequality to estimate this formula:
\begin{equation}
  \Vert u_{R_{k}} \Vert _{H_{\Gamma}^{1}\left(\Omega_{k}\right)}^{2}\leq \Vert u_{R_{0}} \Vert _{H^{1}\left(U\right)}\cdot \Vert u_{R_{k}} \Vert _{H_{\Gamma}^{1}\left(\Omega_{k}\right)}+\boxed{\max\left\{\beta,\lambda\left|U\right|^{\frac{1}{2}}\right\}\cdot\Vert u_{R_{k}} \Vert _{L^{2}\left(\Sigma_{k}\right)}}.
  \label{con:Eq.11}
\end{equation}
\\
\textbf{Incorrectness:}
\begin{equation}
  \int_{\Sigma_{k}}\beta \left|u_{R_{k}}\right|^{2}d\sigma=\beta\Vert u_{R_{k}}\Vert_{L^{2}\left(\Sigma_{k}\right)}^{2},
  \label{con:Eq.12}
\end{equation}
\begin{equation}
  \int_{\Sigma_{k}}\lambda \left|u_{R_{k}}\right|d\sigma\leq \lambda \left(\int_{\Sigma_{k}}1^{2}d\sigma\right)^{\frac{1}{2}}\cdot\left(\int_{\Sigma_{k}}\left|u_{R_{k}}\right|^{2}d\sigma\right)^{\frac{1}{2}}\leq \lambda\left|U\right|^{\frac{1}{2}}\Vert u_{R_{k}} \Vert _{L^{2}\left(\Sigma_{k}\right)},
  \label{con:Eq.13}
\end{equation}
However, the sum of Eq.\ref{con:Eq.12} and Eq.\ref{con:Eq.13} is not equal to the boxed estimation in Eq.\ref{con:Eq.11}, 
as this boxed estimation neglects the squared term in Eq.\ref{con:Eq.12}.
This incorrectness is critical as the following derivation in \cite{1} determines the subsequent compactness proof of $\mathcal{F}$. 
We rectify it through estimation via the Poincaré-Friedrichs inequality, 
which aligns with the coercivity proof for the bilinear form $a+a_{\Sigma}$.
\begin{lemma} \rm (Poincaré-Friedrichs inequality):
Let $\Omega$ be a bounded and connected domain with smooth orientable boundary $\partial\Omega$.
For $\Sigma\subset\partial\Omega$ and $m\left(\Sigma\right)>0$, there exists a positive constant $C$ such that:
\begin{equation}
  \Vert v\Vert_{H^{1}\left(\Omega\right)}\leq C\left(\left|\int_{\Sigma} v d\sigma\right|+\left|v\right|_{H^{1}\left(\Omega\right)}\right),\quad\forall\ v\in H^{1}\left(\Omega\right),
  \label{con:Eq.14}
\end{equation}
where $m$ is the Lebesgue measure and the positive constant $C$ depends only on $\Omega$ and $\Sigma$.
\label{Poincaré-Friedrichs inequality}
\end{lemma}
In the following theorem, We correctly prove boundedness for the solution sequence $u_{R_{k}}$ by employing Lemma \ref{Poincaré-Friedrichs inequality}.
\begin{theorem}(Correction for boundedness)\label{Correction for boundedness}
  There exists a positive constant $C$ such that $\Vert u_{R_{k}}\Vert_{H^{1}\left(\Omega_{k}\right)}\leq C$.
\end{theorem}
\begin{proof}
  For the variational equation Eq.\ref{con:Eq.10}, we rearrange terms in the equation as follows:
  \begin{equation}
  \int_{\Omega_{k}}\nabla u_{R_{k}}\cdot \nabla u_{R_{k}}dx+\int_{\Sigma_{k}}\beta \left|u_{R_{k}}\right|^{2}d\sigma=-\int_{\Omega_{k}}\nabla u_{R_{0}}\cdot \nabla u_{R_{k}}dx+\int_{\Sigma_{k}}\lambda u_{R_{k}}d\sigma.
  \label{con:Eq.15}
  \end{equation}
  The left terms can be formulated by:
  \begin{equation}
  \int_{\Omega_{k}}\nabla u_{R_{k}}\cdot \nabla u_{R_{k}}dx+\int_{\Sigma_{k}}\beta \left|u_{R_{k}}\right|^{2}d\sigma=a\left(u_{R_{k}},u_{R_{k}}\right)+\beta a_{\Sigma}\left(u_{R_{k}},u_{R_{k}}\right).
  \label{con:Eq.16}
  \end{equation}
  The left terms can be estimated via Cauchy-Buniakowsky-Schwarz inequality:
  \begin{equation}
  \int_{\Omega_{k}}\left|\nabla u_{R_{k}}\right|^{2}dx+\int_{\Sigma_{k}}\beta \left|u_{R_{k}}\right|^{2}d\sigma\geq C_{1}\left(\int_{\Omega_{k}}\left|\nabla u_{R_{k}}\right|^{2}dx+\left|\int_{\Sigma_{k}}u_{R_{k}}d\sigma\right|^{2}\right),
  \label{con:Eq.17}
  \end{equation}
  where $C_{1}=\min\left\{1,\beta\cdot m\left(\Sigma_{k}\right)^{-1}\right\}$.
  The quadratic mean can be directly estimated as follows:
  \begin{equation}
  \left(\int_{\Omega_{k}}\left|\nabla u_{R_{k}}\right|^{2}dx+\left|\int_{\Sigma_{k}}u_{R_{k}}d\sigma\right|^{2}\right)\geq\frac{1}{2}\left(\left|u_{R_{k}}\right|_{H^{1}\left(\Omega_{k}\right)}+\left|\int_{\Sigma_{k}}u_{R_{k}}d\sigma\right|\right)^{2}.
  \label{con:Eq.18}
  \end{equation}
  Hence, by the Poincaré-Friedrichs inequality, we derive the following estimation:
  The left terms can be formulated by:
  \begin{equation}
  \int_{\Omega_{k}}\nabla u_{R_{k}}\cdot \nabla u_{R_{k}}dx+\int_{\Sigma_{k}}\beta \left|u_{R_{k}}\right|^{2}d\sigma\geq C_{2}\Vert u_{R_{k}}\Vert_{H^{1}\left(\Omega_{k}\right)}^{2},
  \label{con:Eq.19}
  \end{equation}
  where $C_{2}=\frac{1}{2}C_{1}$.
  Then we formulate this estimation:
  \begin{equation}
  C_{2}\Vert u_{R_{k}}\Vert_{H^{1}\left(\Omega_{k}\right)}^{2}\leq \Vert u_{R_{0}} \Vert _{H^{1}\left(U\right)}\cdot \Vert u_{R_{k}} \Vert _{H_{\Gamma}^{1}\left(\Omega_{k}\right)}+\lambda\left|U\right|^{\frac{1}{2}}\cdot\Vert u_{R_{k}} \Vert _{L^{2}\left(\Sigma_{k}\right)},
  \label{con:Eq.20}
  \end{equation}
  Hence, we use the Cauchy inequality:
  \begin{equation}
  \Vert u_{R_{0}} \Vert _{H^{1}\left(U\right)}\cdot \Vert u_{R_{k}} \Vert _{H_{\Gamma}^{1}\left(\Omega_{k}\right)}+\lambda\left|U\right|^{\frac{1}{2}}\cdot\Vert u_{R_{k}} \Vert _{L^{2}\left(\Sigma_{k}\right)}\leq C_{3}\Vert u_{R_{k}}\Vert_{H^{1}\left(\Omega_{k}\right)},
  \label{con:Eq.21}
  \end{equation}
  where $C_{3}=\left(\Vert u_{R_{0}} \Vert _{H^{1}\left(U\right)}^2+\lambda^{2}m\left(U\right)\right)^{\frac{1}{2}}$.
  That means the boundedness of $u_{R_{k}}$:
  \begin{equation}
  \Vert u_{R_{k}}\Vert_{H^{1}\left(\Omega_{k}\right)}\leq C,
  \label{con:Eq.22}
  \end{equation}
  where $C=C_{3}\cdot C_{2}^{-1}$.
\end{proof}
As above, we have completed the correction of erroneous boundedness estimation in \cite{1}. 
This now can directly use the Theorem \ref{Correction for boundedness}, 
and enables domain extension estimation based on \cite{4} as follows: 
\begin{equation}
  \Vert \tilde{u}_{R_{k}}\Vert_{H^{1}\left(U\right)}\leq \tilde{C}\Vert u_{R_{k}}\Vert_{H^{1}\left(\Omega_{k}\right)}.
  \label{con:Eq.23}
\end{equation}
Subsequent convergence estimations for solution sequences $u_{R_{k}}$ remain correct in \cite{1}. 
Hence, the existence for optimal domains can further be ensured by proving the compactness of $\mathcal{F}$ and the lower semi-continuity of $\mathcal{J}$.

\section{Conclusion}
In this comment, we identify a corrigendum in \cite{1} for the boundedness of sequence ${z_{R_{k}}}$ and present the proof using the Poincaré-Friedrichs inequality. 
This correction is essential for the compactness of $\mathcal{F}$ argument.

\section*{Acknowledgments}
The work is jointly supported by the NSFC and the National Key R\&D Program of China (42330210, 2022YFF0503800, 42004146 and 2021YFA0718600),
and the Specialized Research Fund for State Key Laboratories. 
The authors gratefully acknowledge these supports.









\medskip
Received xxxx 20xx; revised xxxx 20xx; early access xxxx 20xx.
\medskip


\begin{thebibliography}{99}

\bibitem{1}
\newblock J.F.T. Rabago, H. Azegami,
\newblock A new energy-gap cost functional approach for the exterior Bernoulli free boundary problem,
\newblock \emph{Evol. Equ. Control Theory}, \textbf{8(4)} (2019), 785–824.

\bibitem{2}
\newblock H.W. Alt and L.A. Caffarelli,
\newblock Existence and regularity for a minimum problem with free boundary,
\newblock \emph{J. Reine. Angew. Math.},\textbf{325} (1981),105-144.

\bibitem{3}
\newblock A. Boulkhemair, A. Nachaoui and A. Chakib,
\newblock A shape optimization approach for a class of free boundary problems of Bernoulli type,
\newblock \emph{Appl. Math.},\textbf{58} (2013),157-176.


\bibitem{4}
\newblock D. Chenais
\newblock On the existence of a solution in a domain identification problem,
\newblock \emph{J. Math. Anal. Appl.},\textbf{52(2)} (1975),189-219.

\end{thebibliography}
\end{document}